\documentclass[reqno,a4paper]{amsart}
\usepackage{mathrsfs,amssymb,url,color}
\usepackage{cancel} 

\newtheorem{theorem}{Theorem}
\newtheorem{proposition}{Proposition}
\newtheorem{lemma}{Lemma}
\newtheorem{corollary}{Corollary}

\theoremstyle{definition} 
\newtheorem{remark}{Remark}
\newtheorem{definition}{Definition}

\newcommand{\kma}{\mathfrak{g}}
\newcommand{\csa}{\mathfrak{h}}
\newcommand{\Oint}{\mathcal{O}^{\mathrm{int}}}
\newcommand{\cato}{\mathcal{O}}
\newcommand{\sch}{\chi}
\newcommand{\mcgo}{\mathcal{G}_e}
\newcommand{\mcge}{\mathcal{G}^\dag_e}
\newcommand{\mcg}{\mathcal{G}}
\newcommand{\mcj}{\mathcal{J}}
\newcommand{\form}[1][\cdot,\cdot]{(#1)}
\newcommand{\pair}[1][\cdot,\cdot]{\langle #1 \rangle}
\newcommand{\rop}{\Delta_+}
\newcommand{\ro}{\Delta}

\newcommand{\be}{\begin{enumerate}}
\newcommand{\ee}{\end{enumerate}}
\newcommand{\beq}{\begin{equation}}
\newcommand{\eeq}{\end{equation}}
\newcommand{\sr}{\Pi}
\DeclareMathOperator{\ch}{ch}
\DeclareMathOperator{\mult}{mult}

\newcommand{\integers}{\mathbb{Z}}
\newcommand{\complex}{\mathbb{C}}

\DeclareMathOperator{\degree}{degree}

\begin{document}
\title[]{Unique factorization of tensor products for Kac-Moody algebras}
\author{R. Venkatesh and Sankaran Viswanath}
\address{The Institute of Mathematical Sciences\\
CIT campus, Taramani\\
Chennai 600113, India}
\email{rvenkat@imsc.res.in, svis@imsc.res.in}
\begin{abstract}
We consider integrable, category $\mathcal{O}$ modules of indecomposable symmetrizable Kac-Moody algebras. We prove that unique factorization of tensor products of irreducible modules holds in this category, upto twisting by one dimensional modules. This generalizes a fundamental theorem of Rajan for finite dimensional simple Lie algebras over $\mathbb{C}$. Our proof is new even for the finite dimensional case, and uses an interplay of representation theory and combinatorics to analyze the Kac-Weyl character formula.
\end{abstract}
\maketitle

\section{Introduction}
Our base field will be the complex numbers $\complex$ throughout.
\subsection{}In \cite{rajan}, Rajan proved the following fundamental theorem on unique factorization of tensor products:
\begin{theorem}\label{csrthmorig}
Let $\kma$ be a finite dimensional simple Lie algebra, and  $\mathcal{C}$ be the category of finite dimensional $\kma$-modules. 
Let $n,m$ be positive integers and 
$V_1, V_2, \cdots, V_n$ and $W_1, W_2, \cdots, W_m$ be non-trivial irreducible $\kma$-modules in $\mathcal{C}$ such that
$$ V_1 \otimes V_2 \otimes \cdots \otimes V_n \cong  W_1 \otimes W_2 \otimes \cdots \otimes W_m.$$
Then $n=m$ and there is a permutation $\sigma$ of $\{1,2,\cdots,n\}$ such that $V_i \cong W_{\sigma(i)}$ for $1 \leq i \leq n$.
\end{theorem}

The following is an equivalent formulation of theorem \ref{csrthmorig} in which $n=m$, but with trivial modules allowed.
\begin{theorem} \label{csrthm}
Let $\kma$ be a finite dimensional simple Lie algebra, and $\mathcal{C}$ be the category of finite dimensional $\kma$-modules. Let $n$ be a positive integer, and  
suppose $V_1, V_2, \cdots, V_n$ and $W_1, W_2, \cdots, W_n$ are  irreducible $\kma$-modules in $\mathcal{C}$ such that
$$ V_1 \otimes V_2 \otimes \cdots \otimes V_n \cong  W_1 \otimes W_2 \otimes \cdots \otimes W_n.$$
Then there is a permutation $\sigma$ of $\{1,2,\cdots,n\}$ such that $V_i \cong W_{\sigma(i)}$ for $1 \leq i \leq n$.
\end{theorem}
The characters of $\kma$-modules are Weyl group invariant polynomial functions on the maximal torus. Rajan's proof of theorem \ref{csrthmorig} proceeds via an intricate inductive analysis of the characters of tensor products, by fixing one of the variables, and passing to a suitable lower rank Lie algebra.

 Special cases of this theorem for $\kma = \mathfrak{sl}_r$ appear in later works of Purbhoo-Willigenburg (for $n=2$) \cite{pw}, and Bandlow-Schilling-Thi\'ery (for highest weights corresponding to rectangular Young diagrams) \cite{bst}. Their proofs employ the combinatorics of Young tableaux. 

The primary goals of this paper are (i) to give an alternate, elementary proof of theorem \ref{csrthm}, and (ii) to 
obtain a generalization to a natural category of representations of symmetrizable Kac-Moody algebras. 

First, suppose $\kma$ is a Lie algebra (possibly infinite dimensional)  and $\mathcal{C}$ is a category of $\kma$-modules closed under $\otimes$ and containing the trivial module. If the assertion of theorem \ref{csrthm} holds for $(\kma, \mathcal{C})$, then $\mathcal{C}$ cannot contain non-trivial one-dimensional $\kma$-modules. For, if $V$ is a non-trivial one dimensional $\kma$-module in $\mathcal{C}$, then so is $W:=V \otimes V$. But both $V$ and $W$ are irreducible, violating uniqueness (with $V_1=V_2=V$ and $W_1=W, \,W_2=$trivial). 

When $\kma$ is a symmetrizable Kac-Moody algebra, a natural category of representations is $\Oint$, whose objects are integrable $\kma$-modules in category $\cato$. When the 
generalized Cartan matrix of $\kma$ is singular (for example, when $\kma$ is affine), 
we have $\kma/[\kma,\kma] \neq 0$; in other words, there are non-trivial one dimensional $\kma$-modules in $\Oint$. 
Thus, unique factorization of tensor products fails in general for $(\kma, \Oint)$. 
We show that this is essentially the only obstruction, i.e., uniqueness still holds
 upto twisting by one-dimensional representations. The following is the statement of our main theorem:
\begin{theorem}\label{mainthm}
 Let $\kma$ be an indecomposable symmetrizable Kac-Moody algebra.
 Let $n$ be a positive integer and suppose $V_1, V_2, \cdots, V_n$ and $W_1, W_2, \cdots, W_n$ are 
irreducible $\kma$-modules in category $\Oint$ such that 
\begin{equation}\label{maineq}
 V_1\otimes \cdots \otimes V_n\cong W_1\otimes \cdots \otimes W_n.
\end{equation}
Then there is a permutation $\sigma$ of the set $\{1,...,n\}$, and one-dimensional $\kma$-modules $Z_i$ such
 that $V_i \otimes Z_i \cong W_{\sigma(i)}, 1\leq i\leq n$.
\end{theorem}

\begin{remark}
If $\kma$ is finite dimensional, then 
(a) $\Oint = \mathcal{C}$, 
(b) {\em indecomposable} is the same as {\em simple} and
(c) the only one dimensional $\kma$-module is the trivial one.
Thus, theorem \ref{mainthm} is a generalization of theorem \ref{csrthm}.
\end{remark}

\begin{remark}
Unique factorization of tensor products upto twisting by one dimensional modules also appears naturally in the finite dimensional context when $\kma \neq [\kma,\kma]$, for instance, when considering the Lie algebra $\mathfrak{gl}_n$ instead of $\mathfrak{sl}_n$ \cite[theorem 3]{rajan}.
\end{remark}
\begin{remark}
Theorem \ref{mainthm} can be interpreted at the level of characters. For example, 
the characters of finite dimensional irreducible $\mathfrak{sl}_n$-modules are the Schur functions; so if a symmetric polynomial can be factored as a product of Schur functions, then this factorization is unique (cf. \cite[proposition 5.1]{bst} and \cite[theorem 2.6]{pw}).  
Analogously, when $\kma$ is an affine Kac-Moody algebra, the formal characters of irreducible modules in category $\Oint$ form a distinguished basis for the space of theta functions considered as a module over the algebra of holomorphic functions on the upper half plane \cite[chapter 13]{kac}. In this setting, our main theorem implies that if a theta function has a factorization as a product of irreducible characters, then such a factorization is unique.
\end{remark}

\subsection{}
The main idea of our proof of theorem \ref{mainthm} is easily illustrated in the simplest case of $\kma = \mathfrak{sl}_2$.  Here, the highest weights of $V_i$ and $W_j$ are indexed by positive integers $a_i$ and $b_j$. Comparing highest weights of the modules in equation \eqref{maineq}, one obtains $\sum_{i=1}^n a_i = \sum_{j=1}^n b_j$. Taking formal characters on both 
sides and simplifying, one gets:
\beq\label{chareqn}
\prod_{i=1}^n (1-x^{a_i+1}) = \prod_{j=1}^n (1-x^{b_j+1})
\eeq
where $x :=e^{-\alpha}$ and $\alpha$ is the positive root of $\mathfrak{sl}_2$. Note that equation \eqref{chareqn} is essentially just the equality of the product of numerators that appear in the Weyl character formula. It is a classical fact\footnote{see \cite[cor. 2.7]{macd} for an application of this to Poincar\'{e} series of finite Weyl groups.} that equation \eqref{chareqn} implies the equality of the multisets $\{a_i +1: 1 \leq i \leq n\}$ and $\{b_j +1: 1 \leq j \leq n\}$. We recall \cite[proposition 4]{rajan} that one way to prove this is by observing that if these multisets are disjoint (which can be ensured by cancelling common terms in equation \eqref{chareqn}), then $x:=\exp(2\pi i/K)$, where $K$ is the largest element in the union of these multisets, is a zero of exactly one side of equation \eqref{chareqn}. 

Alternatively, we can apply the logarithm to equation \eqref{chareqn} to obtain an equality of formal power series:
\beq\label{logeqn}
\sum_{i=1}^n \sum_{p  >0} \frac{x^{p(a_i+1)}}{p} = \sum_{j=1}^n \sum_{p  >0} \frac{x^{p(b_j+1)}}{p}
\eeq
Now, letting $k$ denote the minimal element in the union of the (disjoint) multisets as above, we observe that (a) all 
terms appearing on both sides of equation \eqref{logeqn} involve only $x^r$ for $r \geq k$ and (b) the term $x^k$ appears on exactly one side of the equation (and with coefficient 1). This is
 the required contradiction.

Our proof of theorem \ref{mainthm} is based on this latter approach. We reinterpret the given isomorphism of tensor products as an equality of products of (normalized) Weyl numerators. These are now power series in $l$-variables, where $l$ is the the rank of $\kma$. We then show that the logarithm of a Weyl numerator has a unique monomial of smallest degree containing all variables (propositions \ref{1}, \ref{keyfact}). This is sufficient to establish uniqueness of the irreducible factors in the tensor product, along the same lines as 
for $\mathfrak{sl}_2$.

We remark that if $\kma$ is a finite dimensional simple Lie algebra, then 
the Weyl numerator is a priori a polynomial, but its logarithm is in 
the larger ring of formal power series. When $\kma$ is infinite dimensional, the Weyl numerator is a formal power series to begin with. 

The paper is organized as follows: \S \ref{prelims} sets up the notational preliminaries, \S \ref{three} contains the key statements concerning the logarithm of normalized Weyl numerators, and the proof of our main theorem, while \S \ref{four} uses an interplay of combinatorics and representation theory to prove the key propositions of \S \ref{three}. The reader who is only interested in the finite dimensional case can skip the discussion concerning non-trivial one dimensional representations of $\kma$, and follow the rest of the paper taking $\kma = [\kma,\kma]=$a finite dimensional simple Lie algebra and $\Oint=$the category of finite dimensional $\kma$-modules.


\section{Preliminaries} \label{prelims}

\subsection{}
We will broadly follow the notations of Kac \cite{kac}.
Let $\kma$ be a symmetrizable Kac-Moody algebra with Cartan subalgebra
$\csa$. Let $\sr$ be the set of simple roots, $\ro$ the set of roots, and $\rop$ the set of positive roots of $\kma$. For $\alpha \in \sr$, let $\alpha^\vee$ denote the corresponding 
simple coroot. Let $P$, $Q$, $P^+$, $Q^+$ be the weight lattice, the
root lattice and the sets of dominant weights and non-negative integer
linear combinations of simple roots respectively. Let $W$ be the Weyl group of $\kma$, generated by the simple reflections $\{s_\alpha: \alpha \in \sr\}$, and let $\varepsilon$ be its sign character.
 Let $\form$ be a nondegenerate, $W$-invariant symmetric bilinear form on $\csa^*$.
Let $\csa'$ be the span of the simple coroots of $\kma$ and 
$\kma' = [\kma,\kma]$ be the derived subalgebra of $\kma$. 
We then have $\kma = \kma' + \csa$ and $\kma' \cap \csa = \csa'$.
The category $\Oint$ consists of integrable $\kma$-modules in category $\cato$ \cite{kac}. The simple objects of $\Oint$ are the highest weight modules $L(\lambda)$ indexed by $\lambda \in P^+$. Given $\lambda \in \csa^*$, define $\overline{\lambda}:=\lambda\mid_{\csa'}$.

Let $\mcg$ be the graph underlying the Dynkin diagram of $\kma$, i.e., $\mcg$ has vertex set $\sr$, with an edge between two vertices $\alpha$ and $\beta$ iff 
$\form[\alpha,\beta] <0$. We will refer to $\mcg$ as the {\em graph of} $\kma$. 
Observe that $\mcg$ does not keep track of the Cartan integers $\frac{2\form[\alpha,\beta]}{\form[\beta,\beta]}$; thus for instance the classical series $A_n, B_n$ and $C_n$ all have the same graph. 

Given $\lambda \in P^+$, the module $L(\lambda)$ has a weight space decomposition $L(\lambda) = \oplus_{\mu \in \csa^*} L(\lambda)_\mu$. The formal character of $L(\lambda)$ is  $\ch L(\lambda)  := \sum_{\mu} \dim(L(\lambda)_\mu)\, e^\mu$.  We define the normalized character by  $\sch_\lambda:=e^{-\lambda}\ch(L(\lambda))$, and the normalized Weyl numerator by:
\beq\label{wnum}
U_\lambda:= e^{-(\lambda + \rho)} \sum\limits_{w\in W}\varepsilon(w)e^{w(\lambda+\rho)}.
\eeq
where $\rho$ is the Weyl vector, satisfying $\pair[\rho,\alpha^\vee]=1$ for all $\alpha \in \sr$.
The Weyl-Kac character formula gives:
\beq\label{wk}
\sch_\lambda= U_\lambda/U_0.
\eeq

\subsection{}
Now, let $X_\alpha:=e^{-\alpha}$, $\alpha \in \sr$ and consider the algebra of formal power series $\mathcal{A}:=\complex[[X_\alpha:\alpha \in \sr\,]]$. Since $L(\lambda)$ has highest weight $\lambda$, it is clear that $\sch_\lambda \in \mathcal{A}$. We also have that $U_\lambda \in \mathcal{A}$, since $(\lambda+ \rho) - w(\lambda + \rho) \in Q^+$ for all $w \in W$. Both $\sch_\lambda$ and $U_\lambda$ have constant term 1.

We call a monomial $\kappa= \prod\limits_{\alpha \in \sr}X^{p_\alpha}_\alpha\in \mathcal{A}$ {\em regular} if $p_\alpha\geq 1$ for all $\alpha\in \sr$.  Given $f \in \mathcal{A}$, say $f = \sum_{\kappa} b_\kappa \, \kappa$ (the sum running over monomials $\kappa$), the regular part of $f$, denoted $f^\#$, is defined to be the sum of only the regular terms in $f$, i.e., $f^\# := \displaystyle\sum_{\kappa \text{ regular}} b_\kappa \, \kappa$. It is easy to see that $f^\#$ is given by the formula
$f^\# = \sum\limits_{J\subset \sr}(-1)^{|J|} \, f|_{{}_{X_\alpha=0,\alpha\in J}}$ but we will not need this.

Also given $\gamma \in P^+$, define the associated regular monomial 
$M^{\gamma}:=\prod\limits_{\alpha \in \sr}X^{\pair[\gamma+\rho,\alpha^{\vee}]}_\alpha$, and let $\deg(\gamma):=\degree (M^\gamma) = \sum\limits_{\alpha \in \sr}  \pair[\gamma+\rho, \alpha^{\vee}]$.

\section{Proof of the main theorem} \label{three}

The following lemma collects together some well-known properties :
\begin{lemma} \label{easy} Let $\kma$ be a symmetrizable Kac-Moody algebra and $\lambda, \mu \in P^+$. 
The following statements are equivalent:
(a) $\sch_\lambda = \sch_\mu$, 
(b) $U_\lambda = U_\mu$,
(c) $M^\lambda = M^\mu$,
(d) $\overline{\lambda} = \overline{\mu}$,
(e) $L(\lambda) \cong L(\mu)$ as $\kma'$-modules,
(f) $L(\lambda) \otimes Z \cong L(\mu)$ as $\kma$-modules, for some one dimensional $\kma$-module $Z$.
\end{lemma}
\begin{proof}
The Weyl character formula (equation \eqref{wk}) shows that (a) and (b) are equivalent, while (c) $\Leftrightarrow$ (d) follows from definitions. The equivalence of (d), (e) and (f) can be found in \cite[\S 9.10]{kac}. The implication (b) $\Rightarrow$ (d) follows from the observation that the only monomial in $U_\lambda$ of the form $X_\alpha^n$ is $-X_\alpha^{\pair[\lambda+\rho, \alpha^\vee]}$ (corresponding to $w = s_\alpha$ in equation \eqref{wnum}). Finally, (d) $\Rightarrow$ (b) because the expression $w(\lambda+\rho) - (\lambda+\rho)$ only depends on the values $\pair[\lambda+\rho, \alpha^\vee]$ for $\alpha \in \sr$ 
(for instance, this follows from equation \eqref{gwg} below and induction).
\end{proof}


Next, recall that if  $\eta \in \mathcal{A}$ is a formal power series with constant term 1, its logarithm is a well defined formal power series: $\log \eta = - \sum_{p \geq 1} (1 - \eta )^p/p$. The next two propositions are the key ingredients in the proof of our main theorem.

\begin{proposition} \label{1}
 Let $\kma$ be a symmetrizable Kac-Moody algebra. Given $\lambda \in P^+$, we have:  
$$(-\log U_\lambda)^\# = c(\kma, \lambda) \,M^\lambda + \text{\rm regular monomials of degree} >\deg(\lambda)$$ 
for some $c(\kma, \lambda) \in \integers_{\geq 0}$. Further, $c(\kma, \lambda)$ is independent of $\lambda$, and only 
depends on the graph of $\kma$.
\end{proposition}

\begin{proposition} \label{keyfact}
Letting $c(\kma):=c(\kma,\lambda)$, we have further that $c(\kma) \geq 1$ iff $\kma$ is indecomposable, or equivalently, iff the graph of $\kma$ is connected.
\end{proposition}
Thus, when $\kma$ is indecomposable, the above propositions imply that $M^\lambda$ is the unique regular monomial of
 minimal degree appearing with nonzero coefficient in $\log U_\lambda$. When $\kma$ is a finite dimensional simple Lie algebra, we will 
in fact  
show (\S \ref{four}) that $c(\kma)=1$.
We defer the proofs of propositions \ref{1} and \ref{keyfact} to section \ref{four}. We first deduce a unique factorization theorem for Weyl numerators (see also \cite[theorem 2]{rajan}), and use this to prove our main theorem.

\begin{theorem}\label{uniqfact}
 Let $\kma$ be an indecomposable symmetrizable Kac-Moody algebra. Let $n,m$ be positive integers and suppose
$\lambda_1,\cdots,\lambda_n,\, \mu_1,\cdots,\mu_m \in P^{+}$ are such that the following identity holds in $\mathcal{A}$:
\begin{equation}\label{ul}
 U_{\lambda_1} \cdots U_{\lambda_n}=U_{\mu_1}\cdots U_{\mu_m}.
\end{equation}
Then $n=m$, and there is a permutation $\sigma$ of the set $\{1,2, \cdots,n\}$, such that $U_{\lambda_i}=U_{\mu_{\sigma(i)}}, 1\leq i\leq n$.
\end{theorem}

\begin{proof}
Let $a:=\min (\{\deg(\lambda_i): 1 \leq i \leq n\} \cup \{\deg(\mu_j): 1 \leq j \leq m\})$. We can assume without loss of generality that $\deg(\lambda_1) = a$.
Now apply the operator $-\log$ to equation \eqref{ul} and consider the regular monomials on both sides :
\beq \label{logulh}
\sum\limits_{i=1}^n (-\log U_{\lambda_i})^\#=\sum\limits_{j=1}^m (- \log U_{\mu_j})^\#. 
\eeq
By propositions \ref{1} and \ref{keyfact}, it is clear that $M^{\lambda_1}$ occurs on the left hand side of equation \eqref{logulh} with nonzero coefficient. Since all $\mu_j$'s have degree $\geq a$, there must exist $1 \leq j \leq m$ for which $M^{\mu_j} = M^{\lambda_1}$. By lemma \ref{easy}, $U_{\lambda_1} = U_{\mu_j}$. Cancelling these terms and proceeding by induction, we obtain the desired conclusion.
\end{proof} 

We now complete the proof of theorem \ref{mainthm}. Given irreducible $\kma$-modules $V_i, W_j$ as in equation \eqref{maineq}, we let $\lambda_i, \mu_j$ be dominant integral weights such that $V_i=L(\lambda_i)$ and $W_j=L(\mu_j)$ for $1 \leq i,j \leq n$. Observe that 
(a) all weights of the module $\bigotimes\limits_{i=1}^n V_i$ are $\leq \sum_{i=1}^n \lambda_i$ where $\leq$ is the usual partial order on the weight lattice, and (b) $\sum_{i=1}^n \lambda_i$ is a weight of  this module. Thus, comparing highest weights of the modules $\otimes_i V_i$ and $\otimes_j W_j$, we conclude $\sum_{i=1}^n \lambda_i = \sum_{j=1}^n \mu_j =: \beta$ (say). Taking formal characters of the modules 
in equation \eqref{maineq} one obtains:
$$\prod_{i=1}^n \ch(L(\lambda_i)) = \prod_{j=1}^n \ch(L(\mu_j))$$
Multiplying both sides by $e^{-\beta} U_0^n$ and using the Weyl-Kac character formula (equation \eqref{wk}), we get  $U_{\lambda_1} \cdots U_{\lambda_n}=U_{\mu_1}\cdots U_{\mu_n}$. Theorem \ref{uniqfact} and lemma \ref{easy} now complete the proof. \qed

\section{Proof of propositions \ref{1} and \ref{keyfact}} \label{four}
Throughout this section, we fix a dominant integral weight $\lambda$ of $\kma$. 
\subsection{}
Let $a_\alpha:=\langle \lambda+\rho, \alpha^{\vee} \rangle \in \integers_{>0}$ for each $\alpha \in \sr$; 
thus $M^{\lambda}:=\prod\limits_{\alpha \in \sr}X^{a_\alpha}_\alpha$. 
We write 
\begin{equation} \label{defn}
\lambda+\rho-w(\lambda+\rho)=\sum\limits_{\alpha\in \sr}c_\alpha(w)\alpha
\end{equation}
where $c_\alpha(w)\in \mathbb{Z}_{\geq 0}$, and define $X(w):=\prod\limits_{\alpha \in \sr}X^{c_\alpha(w)}_\alpha = e^{w(\lambda+\rho) - (\lambda+\rho)}$ .

For $w \in W$, let $\mathbf{w}$ denote a reduced word for $w$. We define 
$I(w):=\{\alpha\in \sr: s_\alpha \text{ appears in } \mathbf{w}\}$; 
this is a well defined subset of $\sr$, since $I(w)$ is independent of the reduced word chosen \cite{humphreys}. 
A non-empty subset $K \subset \sr$ is said to be {\em totally disconnected} if $\form[\alpha,\beta]=0$ for all distinct $\alpha, \beta \in K$, i.e., there are no edges in 
$\mcg$ between vertices of $K$. Let $\mathcal{I}:=\{w\in W\backslash \{e\}: I(w) \text{ is totally disconnected}\}$.  Given a totally disconnected subset $K$ of $\sr$, there is a unique element $w(K) \in \mathcal{I}$ with $I(w(K)) = K$; it is clear that $w(K)$ is just the product of the commuting simple reflections $\{s_\alpha: \alpha \in K\}$. Thus, $\mathcal{I}$ is in natural bijection with the set of all totally disconnected subsets of $\sr$.  Note that the 
elements of $\mathcal{I}$ are involutions in $W$. 
We now have the following key lemma.

\begin{lemma} \label{loglem}
Let $w\in W$. Then
\be
\item[(a)] $I(w)=\{\alpha \in \sr: c_\alpha(w) \neq 0\}$, i.e.,  $X(w)=\prod_{\alpha \in I(w)}X^{c_\alpha(w)}_\alpha$.
\item[(b)] $c_\alpha(w)\geq a_\alpha$ for all $\alpha \in I(w)$.
\item[(c)] If $w \in \mathcal{I}$, then $c_\alpha(w)= a_\alpha$ for all $\alpha \in I(w)$.
\item[(d)] If $w\notin \mathcal{I} \cup \{e\}$, then there exists $\beta \in I(w)$ such that $c_{\beta}(w)> a_\beta$.
\ee
\end{lemma}
\begin{proof}
We set $\gamma:=\lambda + \rho$. First, observe that (c) follows immediately from definitions. Further, equation \eqref{defn} shows that $c_\alpha(w)=0$ for all $\alpha \notin I(w)$. Thus (a) follows from (b). We now prove (b) by induction on $l(w)$. If $w=e$, then (b) is trivially true. Suppose that $l(w) \geq 1$, write $w = \sigma s_{\alpha}$ with $l(\sigma) = l(w) -1$ and $\alpha \in \sr$. This implies $\sigma(\alpha) \in \rop$. Now 
\begin{equation} \label{gwg}
\gamma - w\gamma = (\gamma - \sigma \gamma) + \sigma (\gamma - s_\alpha \gamma) = (\gamma - \sigma \gamma) + a_\alpha \, \sigma \alpha
\end{equation}
Now, either (i) $I(w)=I(\sigma)$ or (ii) $I(w) = I(\sigma) \sqcup \{\alpha\}$. In case (i), we are done by the induction hypothesis. If (ii) holds, observe that $\sigma \alpha = \alpha + \alpha'$ for some $\alpha'$ in the $\integers_{\geq 0}\text{-span of }I(\sigma)$. Equation \eqref{gwg} and the induction hypothesis now complete the proof of (b).

The proof of (d) is along similar lines, by induction on $l(w)$. Observe $l(w) \geq 2$ since $w \notin \mathcal{I} \cup \{e\}$. Write $w = \sigma s_\alpha$ as above. If $\sigma \notin \mathcal{I}$, then the result follows by the induction hypothesis and the fact that $I(\sigma) \subset I(w)$. If $\sigma \in \mathcal{I}$, then clearly $I(w) \neq I(\sigma)$ and so $I(w) = I(\sigma) \sqcup \{\alpha\}$. Since $I(w)$ is not totally disconnected, we must have $\sigma \alpha \neq \alpha$, i.e., $\sigma \alpha = \alpha + \alpha'$ for some {\em non-zero} $\alpha' \in \integers_{\geq 0}\text{-span of }I(\sigma)$. We are again done by (c) and equation \eqref{gwg}.
\end{proof}

\subsection{} \label{abstrgr}
We now make the following useful definition.
\begin{definition}
Let $k$ be a positive integer. A $k$-partition of the graph $\mcg$ is an ordered
 $k$-tuple $(J_1,...,J_k)$ such that 
(a) the $J_i$'s are non-empty pairwise disjoint subsets of the vertex set $\sr$ of $\mcg$,
(b) $\bigcup\limits_{i=1}^k J_i=\sr$, and 
(c) each $J_i$ is a totally disconnected subset of $\sr$. 
\end{definition} 
We let $P_k(\mcg)$ denote the set of $k$-partitions of $\mcg$ and
 $c_k(\mcg) := \mid P_k(\mcg) \mid$. We also define
\begin{equation} \label{cdef}
c(\mcg) := (-1)^l\sum\limits_{k=1}^l (-1)^k \, \frac{c_k(\mcg)}{k}
\end{equation}
 where $l = |\sr|$ is the cardinality of the vertex set of $\mcg$. 
 Finally, given 
$\mcj:=(J_1,\cdots,J_k)$ in  $P_k(\mcg)$, 
define $w(\mcj):=w({J_1}) \,w({J_2}) \cdots w({J_k})$ (this 
is in fact a Coxeter element of $W$).
\subsection{}
We now proceed to analyze $(-\log U_\lambda)^\#$. Write $U_\lambda=1-\xi$, where 
$$\xi:= - \sum\limits_{w\in W\setminus\{e\}} \varepsilon(w) X(w)= \xi_1  + \xi_2$$
 with  $\xi_1 := - \sum\limits_{w\in \mathcal{I}} \varepsilon(w) X(w)$ and 
$\xi_2:= - \sum\limits_{w\notin \mathcal{I}\cup \{e\}}\varepsilon(w)X(w)$.

Since $- \log U_\lambda =\xi +\xi^2/2+...+\xi^k/k+ \cdots$, lemma \ref{loglem} clearly implies that any {\em regular} monomial $\kappa= \prod\limits_{\alpha \in \sr}X^{p_\alpha}_\alpha$ that occurs in $- \log U_\lambda$ must satisfy $p_\alpha \geq a_\alpha$ for all $\alpha \in \sr$. It further implies that there is no contribution of $\xi_2$ to the coefficient of 
the regular monomial $M^\lambda = \prod\limits_{\alpha \in \sr}X^{a_\alpha}_\alpha$, i.e., $M^\lambda$ occurs with the same coefficient in $-\log (1-\xi)$ and in $-\log(1-\xi_1)$.
Thus,  the coefficient of $M^\lambda$ in $-\log(U_\lambda)$ is:
\begin{equation*}
\sum\limits_{k \geq 1}\sum\limits_{\mcj \in P_k(\mcg)} \frac{(-1)^k}{k}\,\varepsilon(w(\mcj))
\end{equation*}
Since $\varepsilon(w(\mcj)) = (-1)^l$ for all $\mcj \in P_k(\mcg)$ and all $k \geq 1$, we deduce that this coefficient is equal to $c(\mcg)$. Thus $c(\kma,\lambda)=c(\mcg)$ only depends on the graph 
$\mcg$ of $\kma$. This establishes all assertions of proposition \ref{1}, except for the non-negative 
integrality of the coefficient $c(\mcg)$. This will be established in proposition \ref{alter} below.

\subsection{} In this section, we obtain another characterization of $c(\mcg)$.
Since $c(\mcg) = c(\kma,\lambda)$ is independent of $\lambda$, we can take $\lambda =0$. Thus, $c(\mcg)$ is the coefficient of $M^0$ in $-\log U_0$. Now, by the Weyl-Kac denominator identity, we have
$$U_0 = \sum_{w \in W} \varepsilon(w) e^{w\rho - \rho} = \prod_{\beta \in \rop} (1 - e^{-\beta})^{\mult \beta}$$
where $\mult \beta$ is the root multiplicity of $\beta$.
So $$-\log U_0 = \sum_{\beta \in \rop} \mult \beta \, \sum_{k \geq 1} \frac{e^{-k\beta}}{k}$$
Since $M^0 = \prod_{\alpha \in \sr} X_\alpha = e^{-\sum_{\alpha \in \sr} \alpha}$, we have thus proved:
\begin{proposition} \label{alter}
$c(\mcg)$ is the multiplicity of the root $\sum\limits_{\alpha \in \sr} \alpha$ in $\kma$.
Thus $c(\mcg) \in \integers_{\geq 0}$.
\end{proposition}
The following statement about roots is well-known, but is included for completeness sake.
\begin{proposition} \label{multprop}
$\sum_{\alpha \in \sr} \alpha$ is a root of $\kma \Leftrightarrow \kma$ is indecomposable. 
\end{proposition}
\begin{proof}
One half ($\Rightarrow$) follows from \cite[Lemma 1.6]{kac}. For the converse, observe that the connectedness of $\mcg$ allows us to order the set $\sr$ of simple roots as $(\alpha_1, \alpha_2, \cdots, \alpha_l)$ such that
the partial sums $\beta_j :=\sum_{i=1}^j \alpha_i$ satisfy $\form[\beta_j,\alpha_{j+1}]<0$ for $1 \leq j <l = |\sr|$. 
Since $(\beta_j - \alpha_{j+1}) \notin \ro$, a standard $\mathfrak{sl}_2$ argument proves that $\beta_{j+1} \in \ro$ if $\beta_j \in \ro$. Since $\beta_1 \in \ro$, we conclude $\beta_l = \sum_{\alpha \in \sr} \alpha$ is also a root. 
\end{proof}

Finally, observe that propositions \ref{alter} and \ref{multprop} prove proposition \ref{keyfact}.

\subsection{} In this subsection, we obtain an algorithm for the computation of $c(\mcg)$, and give an alternate proof of proposition $\ref{keyfact}$. 

We note that the definitions of \S \ref{abstrgr} only need $\mcg$ to be an abstract graph. The main results of this subsection can be viewed as statements about abstract graphs.

\begin{proposition}
Let $\mcg$ be a graph containing at least two vertices, 
and $p$ a vertex of $\mcg$ 
that is adjacent to a unique vertex of $\mcg$. 
Let $\mcg'$ be the graph which is obtained from 
$\mcg$ by deleting the vertex $p$. Then $c(\mcg)=c(\mcg')$.
\end{proposition}
\begin{proof}
Let $q$ denote the unique vertex adjacent to $p$.
Consider $\mcj = (J_1, \cdots, J_k) \in P_k(\mcg)$. 
Then, there are only two (mutually exclusive) possibilities: (a) $J_i=\{p\}$ for some $i$, or (b) $p \in J_i$ for some $i$ for which $|J_i| \geq 2$. We enumerate the number of $k$-partitions of each type. If $\mcj$ is of type (a), then removing $J_i$ gives us a $(k-1)$-partition of $\mcg'$. Thus, the number of $\mcj$'s of type (a) is precisely 
$kc_{k-1}(\mcg')$, since there are $k$ possibilities for $i$.
Next, if $\mcj$ is of type (b), deleting $p$ from the part in which it occurs leaves us with a $k$-partition $\mcj'$ of $\mcg'$. Conversely given $\mcj' \in P_k(\mcg')$, the vertex $p$ can be inserted into any of the $k-1$ parts of $\mcj'$ which do not contain $q$. Thus the number of $\mcj$'s of type (b) is $(k-1)c_k(\mcg')$. Putting these together, we obtain for all $k \geq 1$:
$$c_k(\mcg)=kc_{k-1}(\mcg')+(k-1)c_k(\mcg')$$
where $c_0(\mcg'):=0$. Plugging this into equation \eqref{cdef}, we obtain $$c(\mcg) - c(\mcg') = (-1)^l \sum_{k \geq 1} (-1)^k (c_k(\mcg') + c_{k-1}(\mcg'))$$
where $l$ is the number of vertices in $\mcg$. Since $c_0(\mcg')=0$ and $c_k(\mcg')=0$ for all $k \geq l$, the telescoping sum on the right evaluates to 0.
\end{proof}

\begin{corollary} \label{tree} {\rm (1)} Let $\mcg$ be a tree. Then $c(\mcg)=1$. {\rm (2)} Let $\kma$ be an indecomposable Kac-Moody algebra of finite or affine type with $\kma \neq A_n^{(1)} (n \geq 2)$ (in the notation of Kac \cite{kac}). Then $c(\kma) =1$.
\end{corollary}

\begin{proof}
The first part immediately follows from the above theorem by induction on the number of vertices of $\mcg$. The second follows from the first since for such $\kma$, the associated graph is a tree.
\end{proof}

Next, let $\mcg$ be a graph and $e$ be an edge in $\mcg$. Define 
$\mcge$ to be the graph obtained from
 $\mcg$ by deleting the edge $e$ alone (keeping all vertices, and edges other than $e$ intact). Let $\mcgo$ be the graph which is obtained from $\mcg$ by contraction of the edge $e$ \cite[\S 1.7]{diestel}, in other words, letting $p,q$ denote the vertices at the two ends of $e$, $\mcgo$ is constructed in two steps as follows: (i) delete the vertices $p, q$ of $\mcg$ and all edges incident on them; call this graph $\Gamma$, (ii) create a new vertex $r$; for each vertex $s$ in $\Gamma$, draw an edge between $r$ and $s$ iff $s$ was adjacent to either $p$ or $q$ (or both) in $\mcg$.

\begin{proposition} \label{algo}
With notation as above, $c(\mcg)=c(\mcge)+c(\mcgo)$.
\end{proposition}
\begin{proof}
Let $l$ be the number of vertices in $\mcg$.
 Suppose $\mcj = (J_1, \cdots, J_k)$ is a $k$-partition of $\mcge$. 
Then either (a) $p$ and $q$ occur in different $J_i$'s, or (b) they occur in the same $J_i$. In case (a), $\mcj$ is also a $k$-partition of $\mcg$. In case (b), observe that if we delete $p$ and $q$ from $J_i$ and add $r$ to $J_i$, keeping the remaining $J_p$'s the same,  we obtain a $k$-partition of $\mcgo$. So we get $c_k(\mcg)+c_k(\mcgo)=c_k(\mcge)$ for all $k \geq 1$. From equation \eqref{cdef}, we get 
$$(-1)^l \sum_{k \geq 1} (-1)^k \frac{c_k(\mcg)}{k} = (-1)^l \sum_{k \geq 1} (-1)^k \frac{c_k(\mcge)}{k} + (-1)^{l-1} \sum_{k \geq 1} (-1)^k \frac{c_k(\mcgo)}{k}$$
Since the number of vertices in $\mcgo$ is $l-1$, this proves the proposition. 
\end{proof}

\begin{corollary}
 Let $\kma$ be the affine Kac-Moody algebra of type $A_n^{(1)}, n \geq 2$. Then $c(\kma)=n$.
\end{corollary}
\begin{proof}
This follows from the above theorem since the graph of $\kma$ is an $(n+1)$-cycle. Alternatively, this also follows from 
proposition \ref{alter} since $\sum_{\alpha \in \sr} \alpha$ is just the null root of this affine root system, which has multiplicity $n$.
\end{proof}
Next, we give a purely combinatorial proof of proposition \ref{keyfact}.

\begin{proposition} 
 Let $\mcg$ be a graph.  (i) If $\mcg$ is connected, then $c(\mcg) >0$, and (ii) if $\mcg$ is disconnected, then 
$c(\mcg)=0$.
\end{proposition}
\begin{proof}
 Suppose $\mcg$ is a tree then we are done, since $c(\mcg)=1$. So assume that $\mcg$ contains a cycle, and pick an edge $e$ of this cycle.
Then, $\mcge$ remains connected. It is easy to see that $\mcgo$ is also connected. Both $\mcge$ and $\mcgo$ have strictly fewer edges than $\mcg$. Thus, proposition \ref{algo} together with an induction on the number of edges of $\mcg$ proves (i).  For (ii), suppose there is no edge in $\mcg$, then $\mcg$ has at least two vertices. Let $v$ be a vertex in $\mcg$.
Then since $c_k(\mcg)=k(c_k(\mcg - v)+c_{k-1}(\mcg - v))$, equation \eqref{cdef} gives $c(\mcg)=0$. 
So assume that $\mcg$ contains an edge, and let $e$ be a choosen edge. Then, both $\mcgo$ and  $\mcge$ remain disconnected and have strictly fewer edges than $\mcg$. 
Thus, proposition \ref{algo} together with an induction on the number of edges of $\mcg$ proves the result. 
\end{proof}

\begin{remark}
We note that proposition \ref{algo} gives a recursive algorithm to compute $c(\mcg)$. Since both $\mcge$ and $\mcgo$ have fewer edges than  
$\mcg$, this process terminates in at most $p$ steps, where $p$ is the number of edges in  $\mcg$. In practice, it is even better, terminating as soon as the resulting graphs are trees.
\end{remark}

\subsection{} Finally, putting together the two points of view on $c(\mcg)$, we deduce the following corollary concerning multiplicities of certain roots of symmetrizable Kac-Moody algebras.

\begin{corollary}
Let $\kma$ be an indecomposable symmetrizable Kac-Moody algebra and 
let $\alpha(\kma)$ denote the sum of the simple roots of $\kma$; recall that $\alpha(\kma)$ is a root of $\kma$. Let 
$A=(a_{ij})$  be the generalized Cartan matrix of $\kma$.
\be
\item The root multiplicity of $\alpha(\kma)$ only depends on the graph $\mcg$ of $\kma$. In other words, $\mult \alpha(\kma)$ only depends on the set $\{(i,j): a_{ij} \neq 0\}$ and not on the actual values of the $a_{ij}$.
\item If $\mcge$ and $\mcgo$ are as in proposition \ref{algo} and if $\kma^\dag_e$ and $\kma_e$ are symmetrizable Kac-Moody algebras with graphs $\mcge$ and $\mcgo$ respectively, then 
$\mult \alpha(\kma) = \mult \alpha(\kma^\dag_e) + \mult \alpha(\kma_e)$ \qed
\ee
\end{corollary}


\end{document}